\documentclass{amsart}

\usepackage{graphicx}
\usepackage{verbatim}
\usepackage{amsmath}
\usepackage{amsfonts}

\newtheorem{theorem}{Theorem}[section]
\newtheorem{lemma}[theorem]{Lemma}
\newtheorem{corollary}[theorem]{Corollary}

\theoremstyle{definition}

\theoremstyle{remark}
\newtheorem{remark}[theorem]{Remark}

\numberwithin{equation}{section}

\newcommand{\namedthm}[2]{\theoremstyle{plain}
   \newtheorem*{thm#1}{#1}\begin{thm#1}#2\end{thm#1}}

\def\bcw{\mathbin{\bigcirc\mkern-15mu\wedge}}
\begin{document}

\title{Some Conformally invariant Gap Theorems for Bach-flat 4-manifolds}

\author{Siyi Zhang}
\address{Department of Mathematics, Princeton University, Princeton, New Jersey 08544}
\email{siyiz@math.princeton.edu}






\begin{abstract}
In \cite{CQY}, a conformal gap theorem for Bach-flat metrics with round sphere as model case  was established. In this article, we extend this result to prove conformally invariant gap theorems
for Bach-flat $4$-manifolds with $(\mathbb{CP}^2, g_{FS})$ and $(\mathbb{S}^2\times\mathbb{S}^2,g_{prod})$ as model cases. An iteration argument plays an important role in the case of 
$(\mathbb{CP}^2, g_{FS})$ and the convergence theory of Bach-flat metrics is of particular importance in the case of  $(\mathbb{S}^2\times\mathbb{S}^2,g_{prod})$.
\end{abstract}

\maketitle


\section{Introduction}
In \cite{CGZ}, the author with S.-Y. A. Chang and M. Gursky introduced a conformal invariant $\beta(M,[g])$ and a smooth invariant $\beta(M)$ on a closed, compact Riemannian four-manifold $(M^4,g)$.
These quantities have played a crucial role in the study of conformal geometry in four dimensions; see \cite{CGY03}\cite{CGZ}. For Riemannian four-manifolds, Bach-flatness is a conformal condition which can be viewed as 
a generalization of being Einstein. In \cite{CQY}, a conformal gap theorem for Bach-flat manifolds was proved and its proof was further simplified in \cite{LQS}. The aim of this theorem was to characterize 
the round sphere with conformally invariant conditions. In this article, we generalize this theorem to characterize canonical metrics on $\mathbb{CP}^2$ and $\mathbb{S}^2\times{\mathbb{S}^2}$ 
in a conformally invariant way. We begin our discussion by recalling some definitions and settling down our notations.

Suppose $(M^4,g)$ is a smooth, closed, compact, Riemannian four-manifold. We denote by $Rm$ the curvature tensor, $W$ the Weyl tensor, $Ric$ the Ricci tenor, $R$ the scalar curvature respectively.
We have the well-known decomposition of curvature tensor:
\begin{equation}\label{decomp}
Rm=W+P\bcw{g},
\end{equation}
where $\bcw$ is the Kulkarni-Nomizu product and $P$ is the Schouten tensor defined as
\begin{equation}
P=\frac{1}{2}\left(Ric-\frac{R}{6}\cdot{g}\right).
\end{equation}
With (\ref{decomp}), we can write Chern-Gauss-Bonnet formula in four dimensions as 
\begin{equation} \label{CGB}
8 \pi^2 \chi(M) = \int \|W \|^2 \,dv + 4 \int \sigma_2(g^{-1} P) \,dv,
\end{equation}
where \\
\begin{itemize}
\item $||\cdot||$ is the norm of Weyl tensor viewed as an endomorphism on the bundle of two-forms.  There is another norm $|\cdot|$ of Weyl tensor viewed as a four-tensor. The relation between them is given by
\[||W||^2=\frac{1}{4}|W|^2.\]

\item $g^{-1}P$ is the $(1,1)$-tensor from ``raising'' the subscript of $P$ by contraction with the metric $g$ and $\sigma_2(g^{-1}P)$  
is the second elementary symmetric polynomials applied to the eigenvalues of $g^{-1}P$ viewed as a matrix.
For the sake of simplicity, we will write $\sigma_2(P)$ in place of $\sigma_2(g^{-1}P)$.
\end{itemize}

Given a Riemannian manifold $(M^n, g)$ of dimension $n \geq 3$, let $[g]$ denote the equivalence class of metrics pointwise conformal to $g$, and $Y(M^n,[g])$ denote the Yamabe invariant:
\begin{align*}
Y(M^n,[g]) = \inf_{ \tilde{g} \in [g] } Vol(\tilde{g})^{-\tfrac{n-2}{n}} \int R_{\tilde{g}}\,dv_{\tilde{g}}.
\end{align*}

Recall the following definitions from \cite{CGZ}:
\begin{equation} \label{Y1}
\mathcal{Y}_1^{+}(M^4) = \{ g\ :\ Y(M^4,[g]) > 0 \},
\end{equation}
\begin{equation} \label{Y2def}
\mathcal{Y}_2^{+}(M^4) = \{ g \in \mathcal{Y}_1^{+}(M)\ :\ \int \sigma_2(P_g)\,dv_g > 0 \}.
\end{equation}
For $g\in\mathcal{Y}^{+}_{2}(M)$, define
\begin{equation} \label{betadef}
\beta(M^4,[g]) = \dfrac{ \int \| W_g \|^2 \,dV_g }{ \int \sigma_2( P_g) \, dv_g } \geq 0,
\end{equation}
\begin{equation} \label{betaM}
\beta(M^4) = \inf_{[g]} \beta(M^4,[g]).
\end{equation}
If $\mathcal{Y}_2^{+}(M^4) = \emptyset$, we set $\beta(M^4) = -\infty$.

The conformal invariant $\beta(M,[g])$ can be computed for many Einstein metrics. For example,
\[ \beta(\mathbb{S}^4,[g_c])=0,\,\, \beta(\mathbb{CP}^2,[g_{FS}])=4,\,\, \beta(\mathbb{S}^2\times\mathbb{S}^2,[g_{prod}])=8.\]

In \cite{CQY}, the following conformal gap theorem was proved. See also \cite{LQS}.

\begin{theorem}\label{CQY}
Suppose $(M^4,g)$ is a closed, compact Bach-flat four-manifold. There is an $\epsilon_0>0$ such that if $g\in{\mathcal{Y}_2^+(M^4)}$ with
\begin{equation}\label{pinching S4}
\int\sigma_2(P)dv\geq(1-\epsilon_0)4\pi^2,
\end{equation}
then $(M^4,g)$ is conformally equivalent to $(\mathbb{S}^4,g_{c})$, where $g_{c}$ is the round metric.
\end{theorem}

Theorem \ref{CQY} has the following corollary.

\begin{corollary}
Suppose $(M^4,g)$ is a closed, compact Bach-flat manifold and $b_2(M^4)=0$. There is an $\widetilde{\epsilon}_0>0$ such that if $g\in{\mathcal{Y}_2^+(M^4)}$ with
\begin{equation}\label{pinching S4 prime}
0\leq\beta(M^4,[g])<\widetilde{\epsilon}_0,
\end{equation}
then $(M^4,g)$ is conformally equivalent to $(\mathbb{S}^4,g_{c})$, where $g_{c}$ is the round metric.
\end{corollary}

This corollary shows that if we have some topological restriction ($b_2(M)=0$) on the manifold, then the only Bach-flat metric which is sufficiently close to the round metric in terms of the size of $\beta(M^4,[g])$ is the round metric itself.
The following two theorems generalize this result to the cases of Fubini-Study metric on $\mathbb{CP}^2$ and standard product metric on $\mathbb{S}^2\times\mathbb{S}^2$.

\namedthm{Theorem A}
{Suppose $(M^4,g)$ is an oriented, closed, compact Bach-flat manifold and $b_2^+(M^4)>0$. There is an $\epsilon_1>0$ such that if $g\in{\mathcal{Y}_2^+(M^4)}$ with
\begin{equation}\label{pinching CP2}
4\leq\beta(M^4,[g])<4(1+\epsilon_1),
\end{equation}
then $(M^4,g)$ is conformally equivalent to $(\mathbb{CP}^2,g_{FS})$, where $g_{FS}$ is the Fubini-Study metric.
}

\namedthm{Theorem B}
{Suppose $(M^4,g)$ is an oriented, closed, compact Bach-flat manifold and $b_2^+(M^4)=b_2^-(M^4)>0$. There is an $\epsilon_2>0$ such that if $g\in{\mathcal{Y}_2^+(M^4)}$ with
\begin{equation}\label{pinching S2S2}
8\leq\beta(M^4,[g])<8(1+\epsilon_2),
\end{equation}
then $(M^4,g)$ is conformally equivalent to $(\mathbb{S}^2\times\mathbb{S}^2,g_{prod})$, where $g_{prod}$ is the standard product metric.
}

\section*{acknowledgement}
The author would like to thank his advisor Professor Sun-Yung A. Chang for her constant support. Thanks also to Professor Matthew Gursky for numerous  helpful comments and enlightening discussions.
\section{Preliminaries}
In this section, we recall some basic definitions and collect several preliminary results which will be used throughout this paper. 

In four dimensions, the Bach curvature tensor is defined   as follows:
\begin{equation}\label{Bach tensor}
B_{ij}=\nabla^k\nabla^lW_{kijl}+\frac{1}{2}R^{kl}W_{kijl}.
\end{equation}
The Bach tensor is trace-free, divergence-free and conformally invariant.  It then follows from the conformal invariance that the Bach-flat condition $B_{ij}=0$ is also conformally invariant.
Recall that Bach tensor is the gradient of the Weyl functional $g\to\int{||W||^2}dv$ and Bach-flat metrics are critical points of Weyl functional. Hence, self-dual metrics are Bach-flat 
since they achieve (global) minimum of Weyl functional. By using Bianchi identities, Bach-tensor can be rewritten as:
\begin{equation}\label{Bach tensor rewritten}
B_{ij}=-\frac{1}{2}\Delta{E_{ij}} +\frac{1}{6}\nabla_i\nabla_jR -\frac{1}{24}\Delta{R}g_{ij} - E^{kl}W_{kijl}+E_i^kE_{jk}-\frac{1}{4}|E|^2g_{ij}+\frac{1}{6}RE_{ij},
\end{equation}
where $E_{ij}=R_{ij}-\frac{R}{4}g_{ij}$ is the trace-free Ricci tensor. It is clear from (\ref{Bach tensor rewritten}) that Einstein metrics are Bach-flat.

The following two Weitzenb\"ock identities proved in \cite{CGY02}\cite{CGY03} are of fundamental importance for our discussion.

\begin{lemma}
If $(M^4,g)$ is Bach-flat, then
\begin{equation}\label{Bach-flat E}
\int_M\left(3\left(|\nabla{E}|^2-\frac{1}{12}|\nabla{R}|^2\right)+6tr{E^3}+R|E|^2-6W_{ijkl}E_{ik}E_{jl}\right)dv=0,
\end{equation}
where $tr{E^3}=E_{ij}E_{ik}E_{jk}$, and
\begin{equation}\label{Bach-flat W}
\int_M|\nabla{W^{\pm}}|^2dv=\int_M\left(72\det{W^{\pm}-\frac{1}{2}R|W^{\pm}|^2+2{W^{\pm}_{ijkl}E_{ik}E_{jl}}}\right)dv.
\end{equation}
\end{lemma}

The convergence theory of Bach-flat metrics has been developed by Gang Tian and Jeff Viaclovsky in \cite{TV05a}\cite{TV05b}. The most important ingredients are the following $\epsilon$-regularity theorem and volume estimate.
\begin{theorem}
Let $(M^4,g)$ be a Bach-flat Riemannian four-manifold with Yamabe constant $Y>0$ and $g$ is a Yamabe metric in $[g]$.
Then there exists positive numbers $\tau_k$ and $C_k$ depending on $Y$ such that, for each geodesic ball $B_{2r}(p)$ centered at $p\in{M}$, if
\begin{equation}
\int_{B_{2r}(p)}|Rm|^2dv\leq\tau_k,
\end{equation}
then
\begin{equation}
\sup_{B_{r}(p)}|\nabla^kRm|\leq\frac{C_k}{r^{2+k}}\left(\int_{B_{2r}(p)}|Rm|^2dv\right)^{\frac{1}{2}}.
\end{equation}
\end{theorem}

\begin{theorem}
Let $(X,g)$ be a complete, non-compact, $4$-manifold with base point $p$, and let $r(x)=d(p,x)$, for $x\in{X}$. Assume that there exists a constant $v_0>0$ such that
\begin{equation}
vol(B_r(q))\geq{v_0r^4}
\end{equation}
holds for all $q\in{X}$, assume furthermore that as $r\to\infty$,
\begin{equation}
\sup_{S(r)}|Rm_g|=o(r^{-2}),
\end{equation}
where $S(r)=\partial{B_r(p)}$. Assume that the first Betti number $b_1(X)<\infty$, then $(X,g)$ is an ALE space, and there exists a constant $v_1$ (depending on (X,g)) so that
\begin{equation}
vol(B_r(p))\leq{v_1r^4}.
\end{equation}
\end{theorem}

The last lemma in this section enables us to distinguish Euclidean space from complete non-compact Ricci flat Riemannian manifolds via comparison of the Yamabe constant.
It was proved by Gang Li, Jie Qing and Yuguang Shi in \cite{LQS}.
\begin{lemma}\label{LQS}
Suppose $(M^4,g)$ is a complete non-compact Ricci flat Riemannian manifold. Then $(M^4,g)$ is isometric to the Euclidean space $(\mathbb{R}^4,g_E)$ if
\begin{equation}\label{Yamabe constant gap}
Y(M,g)\geq\eta{Y(\mathbb{R}^4,g_E)}
\end{equation}
for some $\eta>\frac{\sqrt{2}}{2}$.
\end{lemma}
\begin{remark}
This lemma shows that there is a gap phenomenon for the Yamabe constant for non-compact \emph{Ricci flat} Riemannian manifold.
\end{remark}

\section{Modified Yamabe Metric}
In the proofs of main theorems, a variant of Yamabe metric will play a crucial role. In this section, we recall basic properties of this modified Yamabe metric introduced by Matthew Gursky in \cite{Gur00}. 

Let $(M^4,g)$ be a Riemannian four-manifold. Define
\begin{equation}\label{variant of scalar curvature}
F_g^+=R_g-2\sqrt{6}||W_g^+||,
\end{equation}
and
\begin{equation}\label{variant of conformal laplacian}
\mathcal{L}_g=-6\Delta_g+R_g-2\sqrt{6}||W^+||.
\end{equation}
$\mathcal{L}_g$ is a variant of conformal Laplacian that satisfies the following conformal transformation law:
\begin{equation}\label{conformal tranformation law}
\mathcal{L}_{\widetilde{g}}\phi=u^{-3}\mathcal{L}_g(\phi{u}),
\end{equation}
where $\widetilde{g}=u^2g\in[g]$. In analogy to the Yamabe problem, we define the functional
\begin{equation}\label{variant of Yamabe quotient}
\widehat{Y}_g[u]=\left\langle{u,\mathcal{L}_gu}\right\rangle_{L^2}/||u||_{L^4}^2,
\end{equation}
and the associated conformal invariant
\begin{equation}\label{variant of Yamabe constant}
\widehat{Y}(M^4,[g])=\inf_{u\in{W^{1,2}(M,g)}}\widehat{Y}_g[u].
\end{equation}
By the conformal transformation law of $\mathcal{L}_g$, the functional $u\to\widehat{Y}_g[u]$ is equivalent to the Riemannian functional
\begin{equation}\label{variant of definition}
\widetilde{g}=u^2g\to{vol(\widetilde{g})}^{-\frac{1}{2}}\int F_{\widetilde{g}}^+ \,dv_{\widetilde{g}}.
\end{equation}
The motivation for introducing this invariant is explained in the following result (see Theorem 3.3 and Proposition 3.5 of \cite{Gur00}):

\begin{theorem}\label{modified Yamabe}
(i) Suppose $M^4$ admits a metric $g$ with $F_g^+\geq0$ on $M^4$ and $F_g^+>0$ somewhere. Then $b_2^+(M^4)=0$.
(ii) If $b_2^+(M^4)>0$, then $M^4$ admits a metric $g$ with $F_g^+\equiv0$ if and only if $(M^4,g)$ is a K\"ahler manifold with non-negative scalar curvature.
(iii) If $Y(M^4,[g])>0$ and $b_2^+(M^4)>0$, then $\widehat{Y}(M^4,[g])\leq0$ and there is a metric $\widetilde{g}=u^2g$ such that
\begin{equation}
F_{\widetilde{g}}^+=R_{\widetilde{g}}-2\sqrt{6}||W^+||_{\widetilde{g}}\equiv\hat{Y}(M,[g])\leq0
\end{equation}
and
\begin{equation}\label{Gap W+}
  \int R_{\widetilde{g}}^2 \,dv_{\widetilde{g}}\leq 24\int ||W^+_{\widetilde{g}}||^2 \,dv_{\widetilde{g}}.
\end{equation}
Furthermore, equality is achieved if and only if $F_{\widetilde{g}}^+\equiv0$ and $R_{\widetilde{g}}=2\sqrt{6}||W^+_{\widetilde{g}}||\equiv{const}$.
\end{theorem}

\begin{remark}
We shall call the metric realizing $\hat{Y}(M,[g])$ the modified Yamabe metric and denote it by $g_m$. Note that the conditions in (iii) of Theorem {\ref{modified Yamabe}} are satisfied by
the conditions of both Theorem A and Theorem B.
\end{remark}

The following lemma has been established in \cite{CGZ}.

\begin{lemma}\label{initial metric}
Let $(M^4,g)$ be a closed, compact oriented Riemannian four-manifold with $b_2^+(M^4)>0$ and
\begin{equation}\label{CP^2 pinching}
\beta(M^4,[g])=4(1+\epsilon)
\end{equation}
for some $0<\epsilon<1$,
then we have for the modified Yamabe metric $g_m\in[g]$
\begin{equation}\label{g_m W-}
\int_M||W^-_{g_m}||^2dv_{g_m}=\frac{6\epsilon}{2+\epsilon}\pi^2,
\end{equation}
\begin{equation}\label{g_m W+}
\int||W^+_{g_m}||^2dv_{g_m}=12\pi^2+\int||W^-_{g_m}||^2dv_{g_m},
\end{equation}
\begin{equation}\label{Yamabe constant}
Y(M^4,[g])\geq\frac{24\pi}{\sqrt{2+\epsilon}},
\end{equation}
\begin{equation}\label{g_m E}
\int|E_{g_m}|^2dv_{g_m}\leq6\int||W^-_{g_m}||^2dv_{g_m},
\end{equation}
\begin{equation}\label{g_m mu}
\frac{1}{12}\mu_+{Y}\leq3\int||W^-_{g_m}||^2dv_{g_m},
\end{equation}
\begin{equation}\label{g_m R average}
\frac{1}{24}\int(R_{g_m}-\bar{R}_{g_m})^2dv_{g_m}\leq3\int||W^-_{g_m}||^2dv_{g_m},
\end{equation}
where $\bar{R}_{g_m}=\int{R_{g_m}}dv_{g_m}/\int{dv_{g_m}}$.
\end{lemma}

We now prove a similar lemma under the pinching condition of Theorem B.
\begin{lemma}\label{initial metric 2}
Let $(M^4,g)$ be a closed, compact oriented Riemannian four-manifold with $b_2^+(M^4)=b_2^-(M)>0$ and
\begin{equation}\label{S^2S^2 pinching}
\beta(M^4,[g])=8(1+\epsilon)
\end{equation}
for sufficiently small $\epsilon>0$,
then we have for the modified Yamabe metric $g_m\in[g]$
\begin{equation}\label{g_m W+- 2}
\int_M||W^+_{g_m}||^2dv_{g_m}=\int_M||W^-_{g_m}||^2dv_{g_m}=\left(\frac{1+\epsilon}{3+2\epsilon}\right)32\pi^2,
\end{equation}
\begin{equation}\label{Yamabe constant 2}
Y(M^4,[g])\geq16\pi\sqrt{\frac{3}{{3+2\epsilon}}},
\end{equation}
\begin{equation}\label{g_m E 2}
\int|E_{g_m}|^2dv_{g_m}\leq\frac{64\epsilon}{3+2\epsilon}\pi^2.
\end{equation}
\end{lemma}

\begin{proof}
From Corollary F of \cite{Gur98}, $g\in\mathcal{Y}^+_2(M)$ implies that $b_1(M)=0$. Therefore, $\chi=2+b_2(M)=2+2b_2^+(M)$. Then Chern-Gauss-Bonnet and signature formulas imply
\begin{equation}\label{S2S2 CGB}
4\int\sigma_2(P)dv+\int||W||^2dv=8\pi^2(2+2b_2^+(M)),
\end{equation}
\begin{equation}\label{S2S2 sign}
\int||W^+||^2dv=\int||W^-||^2dv.
\end{equation}
The condition $\beta(M,[g])=8(1+\epsilon)$ reads
\begin{equation}\label{beta}
4\int\sigma_2(P)dv=\frac{1}{2(1+\epsilon)}\int||W||^2dv.
\end{equation}
From (\ref{S2S2 CGB})(\ref{S2S2 sign})(\ref{beta}), we solve
\begin{equation}\label{sol}
\int\sigma_2(P)dv=\frac{4\pi^2}{3+2\epsilon}(1+b_2^+(M)),\,\,\, \int||W^{\pm}||^2dv=\frac{(1+\epsilon)16\pi^2}{3+2\epsilon}(1+b_2^+(M)).
\end{equation}
We first show that $b_2^+(M)=1$ if $\epsilon$ is chosen sufficiently small. We now argue by contradiction. Suppose $b_2^+(M)\geq2$. Then from (\ref{sol})
\[\int\sigma_2(P)dv\geq\frac{12\pi^2}{3+2\epsilon}.\]
Theorem \ref{CQY} then implies $(M,g)$ is conformally equivalent to $(\mathbb{S}^4,g_c)$ if $\epsilon$ is chosen sufficiently small. This contradicts to the fact $b_2^+(M)>0$.
Hence, $b_2^+(M)=1$ and $\chi(M)=4$.
Now we have from (\ref{sol})
\begin{equation}\label{sol 2}
\int\sigma_2(P)dv=\frac{8\pi^2}{3+2\epsilon},\,\,\, \int||W^{\pm}||^2dv=\left(\frac{1+\epsilon}{3+2\epsilon}\right)32\pi^2.
\end{equation}
With same argument from the proof of Lemma 2.7 in \cite{CGZ}, we derive
\begin{equation}
\frac{1}{96}Y^2\geq\int\sigma_2(P)dv\geq{\frac{8\pi^2}{3+2\epsilon}}.
\end{equation}
Now we have proved (\ref{g_m W+- 2})(\ref{Yamabe constant 2}). Note that these two inequalities are conformally invariant. 

We now choose the modified Yamabe metric $g_m\in[g]$ and Theorem \ref{modified Yamabe} implies
\begin{equation}\label{modified R^2}
\frac{1}{24}\int{R^2_{g_m}dv_{g_m}}\leq\int||W_{g_m}^+||^2dv_{g_m}=\left(\frac{1+\epsilon}{3+2\epsilon}\right)32\pi^2.
\end{equation}
Recall 
\[\int\sigma_2(P)dv=\frac{1}{96}\int{R_{g_m}^2}dv_{g_m}-\frac{1}{8}\int|E_{g_m}|^2dv_{g_m}.\]
From (\ref{sol 2})(\ref{modified R^2}) we easily derive (\ref{g_m E 2}).
\end{proof}

\section{Proofs of Main Theorems}
\subsection{Proof of Theorem A}
We are now at the position to prove Theorem A. 
Recall from Lemma 2.5 of \cite{CGZ}, we may assume $b_2^+(M)=1$ and $b_2^-(M)=0$ for $\epsilon_1<1$.
The strategy is to choose the modified Yamabe metric and show that the metric is self-dual and Einstein with positive scalar curvature if $\epsilon_1$ is chosen to be sufficiently small. 
By abusing notation, metric in this subsection is the modified Yamabe metric in its conformal class with volume equal to one. From Lemma \ref{initial metric}, we have
\begin{equation}\label{smallness}
\int|E|^2dv<c(\epsilon_1), \int||W^-||^2dv<c(\epsilon_1), \int(R-\bar{R})^2dv<c(\epsilon_1), R-2\sqrt{6}||W^+||=-\mu,
\end{equation}
where $\mu\geq0$ is a constant on the manifold and $\mu\to0$ and $c(\epsilon_1)\to0$ as $\epsilon_1\to0$. Note that $\bar{R}$ is also bounded since the volume is normalized to be one.

We first state an algebraic lemma.

\begin{lemma} \label{W+-EE}
\begin{equation}\label{W+-EE ||W|| and |E|}
  W_{ijkl}^{\pm}E_{ik}E_{jl}\leq\frac{\sqrt{6}}{3}||W^{\pm}|||E|^2
\end{equation}
\end{lemma}
\begin{proof}
This lemma is probably known in the work of \cite{CGY03}\cite{CGZ}. Here we only sketch the proof of this inequality.
 Recall the well-known decomposition of Singer-Thorpe:
\begin{equation}
Riem=
\begin{pmatrix}
W^++\frac{R}{12}Id & B \\
                       B^{*}   &  W^-+\frac{R}{12}Id
\end{pmatrix}
\end{equation}
Note the compositions satisfy
\[BB^*:\Lambda^2_+\to\Lambda_+^2,\quad B^*B:\Lambda_-^2\to\Lambda_-^2.\]

Fix a point $P\in{M^4}$, and let $\lambda_1^{\pm}\leq\lambda_2^{\pm}\leq\lambda_{3}^{\pm}$ denote the eigenvalues of $W^{\pm}$, where ${W^{\pm}}$ are interpreted as endomorphisms of $\Lambda^2_{\pm}$. Also denote the eigenvalues of $BB^*:\Lambda^2_+\to\Lambda_+^2$ by
$b_1^2\leq{b_2^2}\leq{b_3^2}$, where $0\leq{b_1}\leq{b_2}\leq{b_3}$. From similar computation of \cite{CGY03}, we have
\begin{equation}
W_{ijkl}^{+}E_{ik}E_{jl}=4\langle{W^{+},BB^*}\rangle_{\Lambda^2_{+}},\,\,\, W_{ijkl}^{-}E_{ik}E_{jl}=4\langle{W^{-},B^*B}\rangle_{\Lambda^2_{-}}
\end{equation}
From Lemma 6 of \cite{Mar}, we have
\begin{equation}\label{W+-EE}
  W^{\pm}_{ijkl}E_{ik}E_{jl}\leq4\sum_{i=1}^3\lambda_i^{\pm}b_i^2
\end{equation}

Recall from Lemma 4.2 of \cite{CGY03} that $|E|^2=4\sum_{i=1}^3b_i^2$. For a trace-free $3\times3$ matrix $A$, we have the sharp inequality:
\begin{equation}\label{sharp33}
|A(X,X)|\leq{\frac{\sqrt{6}}{3}||A|||X|^2}.
\end{equation}
Apply (\ref{sharp33}) to $A=diag(\lambda_1^{\pm},\lambda_2^{\pm},\lambda_3^{\pm})$ and $X=(b_1,b_2,b_3)$. We derive
\begin{equation}\label{sharp 33}
4\sum_{i=1}^3\lambda_i^{\pm}b_i^2\leq\frac{\sqrt{6}}{3}||W^{\pm}|||E|^2.
\end{equation}
Combining (\ref{W+-EE}) and (\ref{sharp 33}), we derive the desired inequality.

\end{proof}

The basic estimates are from Cauchy-Schwartz inequality:
\begin{equation}
\int|trE^3|dv\leq\left(\int|E|^2dv\right)^{1/2}\left(\int|E|^4dv\right)^{1/2}\leq{c(\epsilon_1)}\left(\int|E|^4dv\right)^{1/2},
\end{equation}
\begin{equation}\label{W-EE}
\int{W^{-}_{ijkl}E_{ik}E_{jl}}dv\leq\left(\int||W^-||^2dv\right)^{1/2}\left(\int|E|^4dv\right)^{1/2}\leq{c(\epsilon_1)\left(\int|E|^4dv\right)^{1/2}},
\end{equation}
\begin{equation}\label{W+EE}
\begin{split}
\int{W^{+}_{ijkl}E_{ik}E_{jl}dv} & \leq\int\left(\frac{\sqrt{6}}{3}||W^+||-\frac{R}{6}\right)|E|^2dv+\int{\frac{R}{6}|E|^2}dv \\
                                                      & =\frac{\mu}{6}\int|E|^2dv+\int{\frac{R}{6}|E|^2}dv
\end{split}
\end{equation}
\begin{equation}\label{R|E|^2}
\begin{split}
\int{R|E|^2}dv &   =\int(R-\bar{R})|E|^2dv+\bar{R}\int|E|^2dv\\
                         &   \leq{\left(\int(R-\bar{R})^2dv\right)^{1/2}\left(\int|E|^4dv\right)^{1/2}+C\int|E|^2dv.} \\
                         &   \leq{c(\epsilon_1)\left(\int|E|^4dv\right)^{1/2}+C\int|E|^2dv.}
\end{split}
\end{equation}
Now apply conformally invariant Sobolev inequality
\begin{equation}
\frac{Y}{6}\left(\int|E|^4dv\right)^{1/2}\leq\int\left(|\nabla{E}|^2+\frac{R}{6}|E|^2\right)dv
\end{equation}
to (\ref{Bach-flat E}). We obtain
\begin{equation}\label{E^4}
\begin{split}
\frac{Y}{6}\left(\int|E|^4dv\right)^{1/2}  &  \leq{\frac{1}{12}\int|\nabla{R}|^2dv-\int{2trE^3dv}-\int\frac{R}{6}|E|^2dv}\\
                                                                    &  +\int{2W_{ijkl}E_{ik}E_{jl}dv} \\
                                                                    &  \leq{c(\epsilon_1)\left(\int|E|^4dv\right)^{1/2}+\frac{1}{12}\int|\nabla{R}|^2dv+C\int|E|^2dv.}
\end{split}
\end{equation}
Since $R-2\sqrt{6}||W^+||=-\mu$ is a constant, we have
\[\nabla{R}=2\sqrt{6}\nabla||W^+||.\]
Integrate over the manifold. We have by Kato inequality
\begin{equation}\label{nabla R}
\int|\nabla{R}|^2dv=24\int|\nabla||W^+|||^2dv\leq{C\int|\nabla{W^+}|^2dv}.
\end{equation}

Recall the sharp estimate
\begin{equation}\label{sharp}
72\det{W^{+}}\leq72\cdot\frac{\sqrt{6}}{18}\cdot\frac{1}{8}|W^+|^3=\frac{\sqrt{6}}{2}|W^+|^3=\frac{R+\mu}{2}|W^+|^2.
\end{equation}

Now apply (\ref{W+EE})(\ref{R|E|^2})(\ref{sharp}) to (\ref{Bach-flat W}). We easily derive
\begin{equation}\label{nabla W}
\int|\nabla{W^{+}}|^2dv\leq\frac{\mu}{2}\int|W^+|^2dv+C\int|E|^2dv+c(\epsilon_1)\left(\int|E|^4dv\right)^{1/2}
\end{equation}

Combining (\ref{E^4})(\ref{nabla R})(\ref{nabla W}), we derive
\begin{equation}
\frac{Y}{6}\left(\int|E|^4dv\right)^{1/2}\leq{C\mu}\int|W^+|^2dv+c(\epsilon_1)\left(\int|E|^4dv\right)^{1/2}+C\int|E|^2dv
\end{equation}

From Lemma \ref{initial metric} and the fact that $c(\epsilon_1)\to0$ as $\epsilon_1\to0$, we can take sufficiently small $\epsilon_1$ to absorb the second term on the right hand side and derive from (\ref{g_m E}) and (\ref{g_m mu})
\begin{equation}\label{E4W2}
\left(\int|E|^4dv\right)^{1/2}\leq{C\int||W^-||^2dv.}
\end{equation}

It is now easy to obtain from (\ref{Bach-flat W}):
\begin{equation}\label{W4}
\begin{split}
\frac{Y}{6}\left(\int|W^-|^4dv\right)^{1/2}  & \leq\int\left(|\nabla{W}^-|^2+\frac{R}{6}|W^-|^2\right)dv \\
                                                                          & =\int\left(72\det{W^-}+2W_{ijkl}^-E_{ik}E_{jl}\right)dv\\
                                                                          &  -\int\left(\frac{R}{3}-\frac{\bar{R}}{3}\right)|W^-|^2dv-\frac{\bar{R}}{3}\int|W^-|^2dv  \\
                                                                          & \leq{c(\epsilon_1)}\left(\int|W^-|^4dv\right)^{1/2}+c(\epsilon_1)\left(\int|E|^4dv\right)^{1/2} \\
                                                                          & -\frac{\bar{R}}{3}\int|W^-|^2dv \\
                                                                          & \leq{c(\epsilon_1)}\left(\int|W^-|^4dv\right)^{1/2}+\left(c(\epsilon_1)-\frac{\bar{R}}{3}\right)\int|W^-|^2dv 
\end{split}
\end{equation}
The first inequality follows from conformally invariant Sobolev inequality. The second inequality is from Cauchy-Schwartz inequality and (\ref{smallness}). The third inequality is from (\ref{E4W2}).
Recall $c(\epsilon_1)\to0$ as $\epsilon_1\to0$ and
\[\bar{R}\geq{Y}\geq\frac{24\pi}{\sqrt{2+\epsilon_1}}.\]
It is then clear to see that we may take sufficiently small $\epsilon_1$ to obtain $W^-\equiv0$ from (\ref{W4}). Recall that $b_2^+(M)=1$. Hence, from Theorem A of \cite{Po}, we can derive $(M^4,g)$ is conformally equivalent to $(\mathbb{CP}^2,g_{FS})$.

\subsection{Proof of Theorem B}
In this subsection, we prove Theorem B. Instead of considering the modified Yamabe metric, we will make use of the Yamabe metric as the conformal representative since the convergence theory for Bach-flat Yamabe metrics has been developed by Tian and Viaclovsky in \cite{TV05a}\cite{TV05b}.

Note that (\ref{g_m W+- 2})(\ref{Yamabe constant 2}) are both conformally invariant. In addition, recall that $\int\sigma_2(P)dv=\frac{1}{96}\int{R^2dv}-\frac{1}{8}\int|E|^2dv$ is conformally invariant and the Yamabe metric realizing the infimum of $\int{R^2}dv$ in a conformal class. It is then clear that
\begin{equation}\label{E^2 Yamabe}
\int|E_{g_Y}|^2dv_{g_Y}\leq\int|E_{g_m}|^2dv_{g_m}\leq\frac{64\epsilon}{3+2\epsilon}\pi^2.
\end{equation}

We now start proving Theorem B and argue by contradiction. Assume that there exists a sequence of Bach-flat $4$-manifolds $(M_j^4,(g_Y)_j)$ which are not conformally equivalent to $(\mathbb{S}^2\times\mathbb{S}^2,g_{prod})$ with $\epsilon_j\to0$, where $(g_Y)_j$ denotes the Yamabe metric with $R_{(g_Y)_j}=4$. 
By (\ref{Yamabe constant 2})(\ref{E^2 Yamabe}), we have

\begin{equation}\label{basic estimate of Yamabe constant}
Y(M_j,(g_Y)_j)\geq16\pi\sqrt{\frac{3}{{3+2\epsilon_j}}}\to16\pi.
\end{equation}
\begin{equation}\label{L-2 Ricci decay}
\int_{M_j}|E_{(g_Y)_j}|^2dv_{(g_Y)_j}\to0.
\end{equation}

Now we first prove that there is no point of curvature concentration. We still argue by contradiction. Note that a uniform lower bound on the Yamabe constants imply a uniform lower bound on the Euclidean volume growth for such sequence of manifolds. Then standing at the point of curvature blow-up, that is, $p_j\in{M_j}$ such that
\begin{equation}\label{blow-up}
\lambda_j=|Rm_{(g_Y)_j}|(p_j)=\max|Rm_{(g_Y)_j}|\to\infty.
\end{equation}
Now consider the sequence of pointed Riemannian manifold $(M_j,g_j,p_j)$ with the re-scaled metric $g_j=\lambda_j^2(g_Y)_j$. Therefore, according to the curvature estimates established in \cite{TV05b}, there exists a subsequence of $(M_j,g_j,p_j)$ which converges to a complete non-compact manifold $(M_\infty,g_\infty,p_\infty)$ in the Cheeger-Gromov topology. It follows that
\begin{itemize}
  \item $Y(M_\infty,g_\infty)\geq16\pi$ and
  \item $Ric(g_\infty)=0$.
\end{itemize}
Here we use similar argument in the proof of Lemma 3.4 of \cite{LQS} to prove $Y(M_\infty,g_\infty)\geq16\pi$. Note that
\begin{equation}\label{Yamabe constant on Euclidean space}
Y(\mathbb{R}^4,g_E)=Y(S^4,g_c)=8\sqrt{6}\pi.
\end{equation}
Hence,
\begin{equation}\label{Yamabe constant lower bound}
Y(M_\infty,g_\infty)\geq\frac{\sqrt{6}}{3}Y(\mathbb{R}^4,g_E).
\end{equation}
Note $\frac{\sqrt{6}}{3}>\frac{\sqrt{2}}{2}$. Hence, Lemma \ref{LQS} implies that $(M_\infty,g_\infty)$ is isometric to Euclidean $4$-space, which is a contradiction to the fact that $|Rm_{g_\infty}|(p_\infty)=1$. Therefore, there is no point of concentration of curvature.

The convergence theory developed in \cite{TV05a}\cite{TV05b} implies that  $(M_j^4,(g_Y)_j)$ converges in $C^\infty$-norms to a smooth Einstein manifold $(\widetilde{M}, \widetilde{g})$ satisfying $\beta(\widetilde{M},[\widetilde{g}])=8$ and $b_2^+(\widetilde{M})=b_2^-(\widetilde{M})=1$.
Hence, from Theorem C of \cite{CGZ}, $(\widetilde{M}, \widetilde{g})$ is (up to rescaling) isometric to $(\mathbb{S}^2\times\mathbb{S}^2,g_{prod})$. Recall that $(\mathbb{S}^2\times\mathbb{S}^2,g_{prod})$ is Bach-rigid in the sense that the metric is an isolated Bach-flat metric 
(see Theorem 7.13 of \cite{GV}). It follows that $(M_j^4,(g_Y)_j)$ must be conformally equivalent to $(\mathbb{S}^2\times\mathbb{S}^2,g_{prod})$ for sufficiently large $j$, which is a contradiction. Therefore, we have proved Theorem B.

\begin{remark}
It is not hard to see that we may prove Theorem A with similar argument as we did in the proof of Theorem B. The current proof of Theorem A has two advantages: first, it avoids making use of the sophisticated general theory of convergence for Bach-flat metrics; second, it is possible to compute a numerical constant of $\epsilon_1$ if we track the coefficients carefully. 
\end{remark}
\begin{remark}
Although it is not possible to get a numerical constant for $\epsilon_2$ from the current proof, we still know that $\epsilon_2$ cannot be too large. On $N=\mathbb{CP}^2\#\overline{\mathbb{CP}}^2$, there is an Einstein metric $g_P$ constructed by D. Page \cite{Pa} satisfying $\beta(N,[g_P])=8.4\dots$. It is then clear that $\epsilon_2<\frac{1}{16}$.
\end{remark}
\bibliographystyle{amsplain}

\end{document}